\documentclass[reqno]{amsart}
\usepackage{amssymb,amsmath,amsthm}
\usepackage{enumitem}
\usepackage{doi}
\usepackage{mathrsfs}

\makeatletter
\@namedef{subjclassname@2020}{\textup{2020} Mathematics Subject Classification}
\makeatother

\theoremstyle{plain}
\newtheorem{theorem}{Theorem}[section]
\newtheorem{lemma}[theorem]{Lemma}
\newtheorem{corollary}[theorem]{Corollary}
\newtheorem{proposition}[theorem]{Proposition}

\theoremstyle{definition}

\newtheorem{remark}[theorem]{Remark}

\newcommand{\F}{\mathbb F_2}
\newcommand{\D}{\mathcal D}
\newcommand{\U}{\mathcal U}
\newcommand{\W}{\mathcal W}
\newcommand{\LL}{\mathcal L}
\newcommand{\Id}{\operatorname{Id}}
\newcommand{\var}{\operatorname{var}}
\newcommand{\aug}{\operatorname{aug}}
\newcommand{\gr}{\operatorname{gr}}

\newcommand{\Mod}{\operatorname{Mod}}

\begin{document}

\title[An explicit limit variety of groups]{An explicit construction of a limit variety of groups}
\author{Zidong Gao}
\address{School of Mathematics, Northwest University, Xi'an, 710127, Shaanxi, P.R. China}
\email{zidonggao@yeah.net}
\author{Yilin Zhou}
\address{School of Mathematics, Northwest University, Xi'an, 710127, Shaanxi, P.R. China}
\email{zhou990619@163.com}


\begin{abstract}
We construct an explicit countable group
\(G_*=N(R)\rtimes A\)
whose generated variety is a limit variety of groups: it has no finite identity
basis, whereas every proper subvariety is finitely based. Here \(A\) is relatively
free in the variety of groups of nilpotency class at most two and exponent
dividing four, and \(R\) is an explicitly defined quotient of the group algebra
\(\mathbb F_2[A]\). The proof establishes an ascending chain condition for
coefficient spaces indexed by six ordered commutator entries and reduces every
proper subvariety to a hereditarily finitely based approximation. Non-finite
basability is obtained by transferring independent identities from the
constructions of Gupta and Krasilnikov. The resulting variety is contained in
\(\mathfrak B_4\mathfrak B_2\). This construction provides the first example of
such a variety in which both the generators and a basis are explicitly given.
\end{abstract}

\maketitle

\section{Introduction}\label{sec:introduction}
A variety of groups is \emph{finitely based} if its identities follow from a finite set of group identities. It is called a \emph{limit variety} if it is not finitely based but every proper subvariety is finitely based. A variety is \emph{Specht}, or \emph{hereditarily finitely based}, if all its subvarieties, including the variety itself, are finitely based. Thus a limit variety is a minimal obstruction to the Specht property.

Every non-finitely based group variety contains a minimal non-finitely based subvariety. Indeed, for a chain of subvarieties ordered by inclusion, their identity theories form a directed union, which is already an equational theory because each inference uses finitely many premises. If all members of a descending chain are non-finitely based, its intersection cannot have a finite basis: such a basis would lie at one stage and force that stage to equal the intersection. Zorn's lemma therefore applies to the non-finitely based subvarieties ordered by reverse inclusion. That argument does not specify a generating group or an identity system for the subvariety obtained. An explicit construction requires both a proof of non-finite basability and control of every additional identity.

Gupta and Krasilnikov developed methods for proving verbal Noetherianity and applied them to varieties of groups in \cite{GK99}. They subsequently constructed a limit variety in the signature of bigroups, namely split groups equipped with a distinguished projection \cite{GK01}. Krasilnikov constructed a non-finitely based variety of groups of exponent eight in \cite{K03}. The distinction between the group and bigroup signatures is essential: forgetting a distinguished projection does not, by itself, transfer either non-finite basability or the limit property.

In this paper we give a concrete group construction in the ordinary group signature. Write $\mathfrak B_n$ for the variety of groups of exponent dividing $n$. As usual, $\mathfrak V\mathfrak W$ denotes the variety of groups having a normal subgroup in $\mathfrak V$ with quotient in $\mathfrak W$.

The main result, stated in full at the end, identifies this explicit ordinary-group variety as a limit variety.

The proof uses the subvarieties $\U_k$ obtained by adjoining the identities
\[
[[x_1,x_2,x_3],y_1^2,\ldots,y_k^2,[z_1,z_2,z_3]]=1.
\]
We prove that every $\U_k$ is Specht and that every proper subvariety of $\U$ is contained in some $\U_k$. A separate argument proves that $\U$ itself is not finitely based. The Noetherian step keeps the six entries of the two triple commutators as ordered indices and uses lifts of actual group substitutions. Higman's lemma \cite{Higman} supplies the required well-quasi-order.

Section~\ref{sec:construction} gives the group and its coefficient algebra. Section~\ref{sec:ambient} proves the Specht property for finitely based ambient varieties containing the $\U_k$. Section~\ref{sec:bigroups} identifies the associated bigroup variety. Sections~\ref{sec:cofinal} and~\ref{sec:nonfinite} treat proper subvarieties and non-finite basability, respectively. Section~\ref{sec:mainproof} completes the proof of Theorem~\ref{main}.

\section{The group and its coefficient algebra}\label{sec:construction}

All commutators are left-normed, with $[x,y]=x^{-1}y^{-1}xy$ and $x^y=y^{-1}xy$. We write $G^\prime$ for the derived subgroup, $\gamma_i(G)$ for the lower central series, and $\mathfrak N_c$ for groups of class at most $c$. The notation $G^m$ denotes the subgroup generated by the $m$th powers. Put $\D=\mathfrak N_2\cap\mathfrak B_4$.

\begin{lemma}\label{freeD}
The relatively free $\D$-group on $a_1,a_2,\ldots$ has unique collected forms
\[
 a_1^{\epsilon_1}a_2^{\epsilon_2}\cdots
 \prod_i(a_i^2)^{\delta_i}\prod_{i<j}[a_i,a_j]^{\eta_{ij}},
 \qquad \epsilon_i,\delta_i,\eta_{ij}\in\{0,1\},
\]
with finite support. Its square subgroup is central elementary abelian on the displayed squares and commutators. Moreover $\D=\var(D_8)$, and $\D\D$ is locally finite.
\end{lemma}
\begin{proof}
In a class-two group of exponent four, the fourth-power formula gives $[a,b]^6\allowbreak=1$; together with $[a,b]^4=1$ this gives $[a,b]^2=1$. Consequently all squares are central. Collection gives existence of the displayed forms. For uniqueness in rank $r$, take a vector space with coordinates $(v,z)$, where $v\in\F^r$ and $z$ has independent coordinates $q_i$ and $c_{ij}$, $i<j$. Set
\[
(v,z)(v',z')=(v+v',z+z'+\sum_i v_iv'_i q_i+
\sum_{i<j}v_jv'_i c_{ij}).
\]
The correction is bilinear, so this is an associative group of class at most two and exponent four. Its $i$th generator squares to $q_i$ and its $i,j$ commutator is $c_{ij}$. Thus none of the collected coordinates can be identified. Passing through finite ranks proves the assertion.

Every word modulo the identities of $\D$ collects into generator powers and pair commutator powers. If it is an identity of $D_8$, setting all but one variable equal to $1$ makes every generator exponent divisible by four. Setting all but a pair equal to $1$, and then using a noncommuting pair in $D_8$, makes every commutator exponent even. Hence it holds in every $\D$-group.

The same forms bound the order of an $r$-generated $\D$-group by $4^r2^{\binom r2}$. If $H$ is finitely generated with $K\lhd H$ and $K,H/K\in\D$, then $H/K$ is finite. The usual Schreier rewriting with a finite transversal supplies finitely many generators for $K$; hence $K$ is finite too. This proves the last assertion. The class of these extensions is a variety: it is defined by substituting words from the verbal subgroup defining the quotient variety into identities of the kernel variety. Thus containment of a generating group in $\D\D$ implies containment of its generated variety.
\end{proof}

Let $A$ be the relatively free countable-rank $\D$-group. In $\F[A]$ define
\[
t(a,b)=[a,b]+a^2b^2+a^2+b^2,\qquad
T=\langle t(a,b):a,b\in A\rangle,\quad R=\F[A]/T.
\]
\begin{lemma}\label{ringbasis}
Inversion induces an involutive anti-automorphism $\mu$ of $R$, and every endomorphism of $A$ acts on $R$. With $s_i=1+a_i^2$, one has
\begin{equation}\label{pbw}
R=\bigoplus_{S\subset\mathbb N\text{ finite}}a^S\mathscr S,
\qquad \mathscr S=\F[s_1,s_2,\ldots]/(s_i^2:i\ge1),
\end{equation}
where $a^S$ is the increasing product. The $s_i$ are central, and, for all $a,b\in A$,
\begin{equation}\label{ring-laws}
s(a)^2=0,\qquad s(ab)=s(a)+s(b),\qquad [a,b]=1+s(a)s(b),
\quad s(a)=1+a^2.
\end{equation}
In particular $(a,b,c)_{\rm ass}=0$ in $R$ for all $a,b,c\in A$.
\end{lemma}
\begin{proof}
Every summand of $t(a,b)$ is a central involution in $A$, so inversion fixes $t(a,b)$ and preserves the two-sided ideal they generate. Inversion on a group algebra reverses products and squares to the identity, which proves the assertion about $\mu$. Endomorphism invariance follows by mapping $t(a,b)$ to $t(\theta(a),\theta(b))$.

Here is a proof of the basis assertion, also explaining the ring calculation used in GK01, Lemmas 3--8. Write $Q=A^2$, $z_{ij}=1+[a_i,a_j]$. Lemma~\ref{freeD} identifies
\[
\F[Q]=\F[s_i,z_{ij}:i<j]/(s_i^2,z_{ij}^2),\qquad
\F[A]=\bigoplus_S a^S\F[Q].
\]
Put $f_{ij}=z_{ij}+s_i s_j=t(a_i,a_j)$. Replacing each $z_{ij}$ by $f_{ij}$ is an invertible change of generators in this commutative square-zero polynomial algebra. Let $T_0$ be the ideal generated by all $f_{ij}$. It follows immediately that $\F[A]/T_0$ has the basis in \eqref{pbw}. We check $T=T_0$, so that no additional relations have been overlooked.

In $\F[A]/T_0$, a product $g=a_{i_1}\cdots a_{i_m}$ of distinct generators satisfies $g^2=1+\sum_{j=1}^m s_{i_j}$. Inductively, if $h$ is the first $m-1$ factors and $S=\sum_{j<m}s_{i_j}$, then
\[
\begin{aligned}
\relax[h,a_{i_m}]&=\prod_{j<m}(1+s_{i_j}s_{i_m})=1+S s_{i_m},\\
(ha_{i_m})^2&=(1+S)(1+s_{i_m})(1+S s_{i_m})=1+S+s_{i_m}.
\end{aligned}
\]
Multiplication by an element of $Q$ changes no square. Thus $s(g)$ is the sum of $s_i$ over the parity support of $g$ for every $g\in A$, proving $s(gh)=s(g)+s(h)$. The class-two power formula then gives
\[
[g,h]=(gh)^2g^2h^2=(1+s(g)+s(h))(1+s(g))(1+s(h))=1+s(g)s(h).
\]
Consequently every $t(g,h)$ vanishes modulo $T_0$, and $T=T_0$. This proves the basis and all the claimed square relations. Finally
\[
(a,b)_{\rm ass}=ba\,s(a)s(b),\qquad
(a,b,c)_{\rm ass}=cba\,s(a)s(b)(s(a)+s(b))s(c)=0.
\]
The augmentation map sends every generator of $T$ to zero and $1$ to $1$; thus it descends to $R$, which is nonzero.
\end{proof}

Put $P=(s_1,s_2,\ldots)\lhd R$. Define on $R^3$ the product
\begin{equation}\label{heis}
(u,v,c)(u',v',c')=(u+u',v+v',c+c'+u\mu(v')).
\end{equation}
\begin{lemma}\label{heisfacts}
This product defines a group $N(R)\in\D$. The formulas $(u,v,c)^a=(ua,va,c)$ define a right action of $A$ by automorphisms. The resulting group
\begin{equation}\label{groupstar}G_*=N(R)\rtimes A\end{equation}
is countable, belongs to $\D\D\cap\mathfrak B_4\mathfrak B_2$, and satisfies
\[
N(R)^2=N(R)'=Z(N(R))=\{(0,0,c):c\in R\}\le Z(G_*).
\]
Its generated variety is locally finite.
\end{lemma}
\begin{proof}
For three factors the central coordinate in either associative bracketing is
$c+c'+c''+u\mu(v')+u\mu(v'')+u'\mu(v'')$. The identity is $(0,0,0)$ and the inverse of $(u,v,c)$ is $(u,v,c+u\mu(v))$. Therefore
\[
(u,v,c)^2=(0,0,u\mu(v)),\quad
[(u,v,c),(u',v',c')]=(0,0,u\mu(v')+u'\mu(v)).
\]
These formulas give class at most two and exponent dividing four. Every $(0,0,c)$ is both the square of $(c,1,0)$ and the commutator of $(c,0,0)$ with $(0,1,0)$. Centralizing $(0,1,0)$ forces $u=0$; centralizing $(1,0,0)$ forces $v=0$. This proves the equality of the three subgroups.

The action law follows from right multiplication in $R$. It preserves the product because $(ua)\mu(v'a)=ua\mu(a)\mu(v')=u\mu(v')$, and it fixes the central coordinate subgroup pointwise. Both factors are in $\D$. For $q\in A^2$, the square of $nq$, with $n=(u,v,c)$, lies in $N(R)$ and has first coordinates $u(1+q),v(1+q)$. Its square is trivial because $\mu(q)=q$ and $(1+q)^2=0$. Hence $N(R)\rtimes A^2$ is a normal subgroup of exponent dividing four with elementary abelian quotient. Countability follows from finite-support words and sums; local finiteness follows from Lemma~\ref{freeD} and membership in $\D\D$.
\end{proof}

Write $\U=\var_{\rm grp}(G_*)$ and
\[
b_k=[[x_1,x_2,x_3],y_1^2,\ldots,y_k^2,[z_1,z_2,z_3]],\qquad
\U_k=\U\cap\Mod(b_k=1).
\]
Here $k\ge0$; $b_0=1$ says that $\gamma_3$ is abelian, since conjugates of triple values are again triple values and these generate $\gamma_3$. The notation $\Id(\mathcal V)$ means the full group identity theory of $\mathcal V$.

\section{Hereditarily finitely based ambient varieties}\label{sec:ambient}
\subsection{Structural identities}
For group words set
\[
t(\mathbf x)=[x_1,x_2,x_3],\qquad
r(h;a,b)=h^{[a,b]}h^{a^2b^2}(h^{a^2})^{-1}(h^{b^2})^{-1}.
\]
For $k\ge1$, let $\W_k$ be defined by the following six identities, with all displayed variable lists independent:
\begin{align}
x^8&=1,\label{W1}\\
[t(\mathbf x)^2,z]&=1,\label{W2}\\
[[t(\mathbf x),t(\mathbf y)],z]&=1,\label{W3}\\
[[x^4,t(\mathbf y)],z]&=1,\label{W4}\\
[r(t(\mathbf x);a,b),z]&=1,\label{W5}\\
b_k&=1.\label{W6}
\end{align}
We shall prove in Theorem~\ref{Wspecht} that $\W_k$ is Specht. The finite defining system will convert finite relative identity bases into absolute finite bases.

\begin{lemma}\label{structure}
Let $G\in\W_k$ be countably generated and let $\Gamma=\gamma_3(G)$. Then
\[
\Gamma^2\le Z(G),\quad \Gamma'\le Z(G),\quad \exp\Gamma'\le2.
\]
Moreover $M=\Gamma/(\Gamma\cap Z(G))$ is a right module over a quotient of $R$; equivalently it is an $\F[A]$-module annihilated by $T$, after choosing an epimorphism from the free $\D$-group onto the acting group.
\end{lemma}
\begin{proof}
The values of $t$ and their conjugates generate $\Gamma$. Equations \eqref{W2}--\eqref{W3} first give central commutators between any two such values. Thus $\Gamma$ has class at most two. For two such values $u,v$, the identity $[u^2,v]=[u,v]^2$ gives exponent two for their commutator. Collecting a product of these values proves the assertions about all of $\Gamma$.

Conjugation on $M$ factors through $G/(G^4\Gamma)$, by \eqref{W4} and centrality of $\Gamma'$. This quotient has class at most two and exponent at most four. Equation \eqref{W5}, viewed in $M$, says precisely that the generators of $T$ annihilate each triple value. Conjugation and addition extend this to all of $M$.
\end{proof}

The alternating pairing
\[
\beta:\Gamma\times\Gamma\longrightarrow\Gamma',\qquad \beta(u,v)=[u,v],
\]
is $\F$-bilinear modulo central elements. Centrality of its values gives the balancing rule
\begin{equation}\label{balance}
\beta(u^f,v^g)=\beta(u^{f\mu(g)},v),\qquad f,g\in R.
\end{equation}
Group-ring exponentiation in a pairing is well-defined even when a chosen product representing $u^f$ is only determined modulo the center.

\subsection{A coefficient cover and substitution operators}\label{subsec:cover}
Let $F$ be the relatively free group of countable rank in $\W_k$, and $\Gamma=\gamma_3(F)$. Define the $\F$-vector space
\[
D_k=\bigoplus_{\tau\in\mathbb N^6}e_\tau(R/P^k).
\]
The quotient by $\Gamma\cap Z(F)$ used below is compatible with every endomorphism of $F$. Indeed, if a central word $w$ involves only $x_1,\ldots,x_n$, then $[w,x_j]=1$ for a fresh generator $x_j$. Substituting arbitrary images for $x_1,\ldots,x_n$ and an arbitrary element for $x_j$ shows that every endomorphic image of $w$ is central. Thus $Z(F)$, and hence $\Gamma\cap Z(F)$, is fully invariant here.

There is a surjective linear map
\begin{equation}\label{cover}
\pi(e_{(a,b,c,d,e,f)}q)
=\bigl[[x_a,x_b,x_c]^q,[x_d,x_e,x_f]\bigr].
\end{equation}
Indeed, basic triple commutators normally generate $\Gamma$, and its commutator pairing is central and bilinear. Relation \eqref{W6} kills $P^k$ in \eqref{cover}. It also kills every coefficient multiple of $P^k$: first conjugate the independent first triple slot in that identity, and only then identify any repeated variable indices.

\begin{remark}
The inverse image under $\pi$ of a verbal subgroup need not be an $R$-submodule. We therefore work with invariant linear subspaces for lifts of group substitutions, rather than with coefficient submodules.
\end{remark}

Let $\Phi$ be the increasing reindexings $x_i\mapsto x_{\phi(i)}$, with $\phi:\mathbb N\to\mathbb N$ strictly increasing. Let
\[
\psi_{ij}:x_j\mapsto x_i x_j,\quad x_\ell\mapsto x_\ell\ (\ell\ne j),\qquad i<j.
\]
They act on $F$. Increasing reindexings have the evident lifts to $D_k$.

To lift $\psi_{ij}$, write $t(a,b,c)=[a,b,c]$ additively in $\Gamma/(\Gamma\cap Z(F))$. Ordinary commutator collection gives
\begin{align}
t(ab,c,d)&=t(a,c,d)^{[b,c]}+t(a,c,b)^{(d-1)[b,c]}+t(b,c,d),\label{tri1}\\
t(a,bc,d)&=t(a,c,d)^{[a,b]}+t(a,b,d)+t(a,b,c)^{d-1},\label{tri2}\\
t(a,b,cd)&=t(a,b,d)+t(a,b,c)^d.\label{tri3}
\end{align}
For example $[a,c]^b=[a,c][a,c,b]$; commuting the two resulting factors with $d$ proves \eqref{tri1}. The omitted conjugations by members of $\Gamma$ are central and hence zero in this quotient.

Order triples by their ordinary ordered lexicographic order, without sorting their entries. Equations \eqref{tri1}--\eqref{tri3} show
\[
\psi_{ij}(t_\rho)=\sum_{\sigma\le\rho}t_\sigma^{\alpha_\sigma}.
\]
Every non-original triple decreases at the first changed slot. Repeated occurrences of $j$ are handled by expansion slot by slot. If $j$ does not occur in $\rho$, the expression is simply $t_\rho$.
In these expansions the coefficient of the unchanged marker triple is $1$ modulo $P$ (in particular, every group-commutator factor is $1$ modulo $P$). All coefficients have nonnegative $P$-degree.

If the two triple expansions have coefficients $\alpha_\sigma,\beta_\nu$, choose the lift
\begin{equation}\label{lift}
e_{(\rho,\eta)}q\ \longmapsto\
\sum_{\sigma\le\rho,\,\nu\le\eta}
e_{(\sigma,\nu)}\alpha_\sigma\psi_{ij}(q)\mu(\beta_\nu).
\end{equation}
This commutes with \eqref{cover} by \eqref{balance}. These operators preserve the finite $P$-filtration of $D_k$. The argument below requires stability only under the displayed linear operators.

\subsection{The ascending chain condition}\label{subsec:acc}
\begin{lemma}[Finite-alphabet subsequence lemma]\label{higmanfinite}
The finite words over a finite alphabet, ordered by equal-letter subsequence embedding, form a well-quasi-order: every infinite sequence contains two terms in their original order with the first embedding in the second.
\end{lemma}
\begin{proof}
This is the finite-alphabet case of Higman's lemma \cite{Higman}; the following proof makes the needed case independent of a further citation. If a bad infinite sequence exists, choose its words successively so that at each position the word has least possible length among extensions of the chosen prefix to a bad sequence. No chosen word is empty. Infinitely many chosen words $w_{i_0},w_{i_1},\ldots$ start with the same letter $a$, say $w_{i_j}=a v_j$. The sequence
\[
w_0,\ldots,w_{i_0-1},v_0,v_1,\ldots
\]
is also bad. Indeed a comparison within its first block contradicts the original badness; a comparison from that block to some $v_j$ gives the same comparison to $w_{i_j}$; and $v_j\le v_\ell$ gives $a v_j\le a v_\ell$. This contradicts the minimal length chosen for $w_{i_0}$, since $v_0$ is shorter. Thus no bad sequence exists.
\end{proof}

\begin{theorem}\label{marked}
The vector space $D_k$ satisfies the ascending chain condition on linear subspaces stable under the above lifts of $\Phi$ and all $\psi_{ij}$ with $i<j$.
\end{theorem}
\begin{proof}
Fix a filtration degree $d<k$. By \eqref{pbw}, a basis of $\gr_P^dD_k$ is
\begin{equation}\label{graded-basis}
e_\tau a^S s_I,\qquad \tau\in\mathbb N^6,\quad S,I\subset\mathbb N\text{ finite},\quad |I|=d.
\end{equation}
Here $s_I=\prod_{i\in I}s_i$, and $S$ need not be disjoint from $I$ or the entries of $\tau$.

Order these basis elements first by ordinary lexicographic order of $\tau$, then by lexicographic order of the increasing $d$-tuple $I$, and then by $(|S|,\operatorname{colex}S)$. For equal cardinalities, colex compares at the largest element of the symmetric difference. This is a well-order, preserved by increasing reindexings. We call the largest basis term of a nonzero finite vector its head.

Suppose a head contains a position
\[
j\in S\setminus\bigl(\{\text{entries of }\tau\}\cup I\bigr),
\]
and choose $i<j$ absent from the support of the entire vector. Applying $\psi_{ij}$ changes this head exactly to
\begin{equation}\label{inflate}
e_\tau a^{S\cup\{i\}}s_I.
\end{equation}
To check this, the marker tuple is fixed because it does not contain $j$. Also $I$ is fixed. On the associated graded coefficient ring, $a_j$ changes to $a_i a_j$. Commuting ordered $a$-factors produces only higher $P$-degree corrections. The possible factors from $\mu$ likewise do not lower $P$-degree.

All lower terms remain lower. A smaller marker tuple cannot increase, by \eqref{tri1}--\eqref{tri3}. For the same tuple, a smaller $I$ cannot increase: $s_j\mapsto s_i+s_j$, and replacing $j$ by $i<j$ decreases the sorted tuple. For the same $\tau,I$, terms not containing $j$ do not gain cardinality, while those containing $j$ gain the same fresh $i$. Degree followed by colex therefore preserves the strict head comparison.

It remains to prove a well-quasi-order for this head expansion. Encode every position from $1$ through the largest index occurring in $\tau,I,S$, including positions carrying no variable. Positions in $\tau\cup I$ are special; record their marker-slot subset, their $I$-flag, and their parity bit (membership in $S$). Ordinary positions are colored $0$ or $1$ according to their membership in $S$. There are at most $6+d$ special positions. In every gap before, between, or after the special positions that contains a non-special parity-$1$ position, mark the last such position as its pivot. Give pivots the color of their gap, distinct from ordinary-position colors.

Partition these words by the ordered special-position pattern, its parity bits, and the presence or absence of a pivot in each gap. There are only finitely many patterns, and each uses a finite alphabet. Lemma~\ref{higmanfinite} applies to the equal-color subsequence order in each class. Because the numbers and ordered colors of all special positions and pivots are fixed, an embedding matches every target special position and every target pivot.

Extend such a subsequence embedding to an increasing $\phi:\mathbb N\to\mathbb N$; place the tail beyond the largest target position. Every additional target parity-$1$ position lies before its matched gap pivot and outside $\phi(\mathbb N)$. It is consequently fresh for \emph{the entire source vector}, even if a lower source term has an index larger than the head's largest index. Reindex by $\phi$, then insert each missing parity-$1$ position using \eqref{inflate}, with the matched pivot as target. These are actual allowed operators, and they preserve the head at every step.

Thus the heads of an invariant subspace form an upward-closed subset of a well-quasi-order. Choose finitely many vectors whose heads generate this upward-closed set. Any vector in the subspace has its head cancelled by a transformed chosen vector. The well-order of heads makes repeated cancellation terminate. This proves finite generation of each invariant subspace in the fixed graded layer, and hence its ACC.

There are only $k$ filtration layers. The usual intersection-and-image argument for a finite invariant filtration transfers ACC from those layers to $D_k$.
\end{proof}

\begin{corollary}\label{centralacc}
The central subgroup $\Gamma'$ satisfies ACC on subgroups that are fully invariant in $F$.
\end{corollary}
\begin{proof}
Their inverse images under \eqref{cover} are invariant linear subspaces for the actual lifts just considered. Apply Theorem \ref{marked}.
\end{proof}

\subsection{The Specht property}\label{subsec:specht}
We spell out the external result and its parameter. For elements of a group ring put
\[
(a,b)_{\rm ass}=ab-ba,\qquad
(a,b,c)_{\rm ass}=abc-bac-cab+cba.
\]
When $V$ is an abelian normal subgroup, $v^{\sum n_g g}$ means $\prod(v^g)^{n_g}$; a product of ring operators acts successively. This is the convention needed below, rather than the group commutator convention for square brackets.

\begin{theorem}[Gupta--Krasilnikov]\label{GK99exact}
Let $E$ be the free group of countably infinite rank. For an integer $r\ge2$, let $V_r$ be its verbal subgroup determined by
\begin{equation}\label{GK99laws}
[[x_1,x_2,x_3],[x_4,x_5,x_6]]=1,\qquad
[x_1,x_2,x_3]^{(x_4,x_5,x_6)_{\rm ass}\cdots
(x_{3r-2},x_{3r-1},x_{3r})_{\rm ass}}=1.
\end{equation}
Then $E/V_r$ satisfies ACC on normal subgroups invariant under increasing reindexings of its free generators. In particular it satisfies ACC on fully invariant subgroups, and the variety defined by \eqref{GK99laws} is Specht.
\end{theorem}
\begin{proof}[Source and specialization]
This is \cite[Propositions 1--2, p.~332]{GK99}, using its parameter $k=r$ and renaming its verbal subgroup $U_k$ to $V_r$ to avoid confusion with our $\U_k$. The paper states the result for all positive parameters; only $r\ge2$ is needed here. Proposition 2 is the stronger assertion about increasing reindexings, and Proposition 1 is its fully invariant consequence. The proof is on pp.~353--355: the third lower central subgroup is treated as a module over the free class-two quotient; the module chain condition is then combined with the chain condition on the quotient. The two displayed identities form a finite defining system. ACC gives a finite relative basis for each subvariety, and adjoining this defining system gives an absolute finite basis. Thus the conclusion concerns ordinary group identities, with no exponent-four hypothesis or distinguished projection added. There are $r-1$ triple factors. In particular, $r=3$ gives exactly
\[
[x_1,x_2,x_3]^{(x_4,x_5,x_6)_{\rm ass}(x_7,x_8,x_9)_{\rm ass}}=1.
\]
Here the source identifier is GK99, DOI 10.1017/S1446788700002056, theorem identifiers Propositions 1 and 2; no arXiv identifier has been established.
\end{proof}

\begin{lemma}\label{quotient99}
For $F$ as above, $F/\Gamma'$ belongs to the Specht variety in Theorem~\ref{GK99exact} with $r=3$.
\end{lemma}
\begin{proof}
Let $Q=F/\Gamma'$ and $V=\gamma_3(Q)=\Gamma/\Gamma'$. Then $V$ is abelian, so ring-operator actions on it are well-defined. Lemma~\ref{structure} gives the further quotient
$M=\Gamma/(\Gamma\cap Z(F))$, annihilated by $T$. By Lemma~\ref{ringbasis},
\[
(a,b,c)_{\rm ass}=cba\,s(a)s(b)(s(a)+s(b))s(c)=0
\quad\hbox{in }R.
\]
For arbitrary $h\in\Gamma$ and $a,b,c\in F$, it follows that the first Lie-triple action on $h\Gamma'$ lies in $(\Gamma\cap Z(F))/\Gamma'$, which is central in $Q$. It is not necessary, and has not been asserted, that this first action is zero in $V$. Every central element $z$ satisfies $z^\alpha=z^{\aug(\alpha)}$ for a group-ring operator $\alpha$. Since an associative Lie triple has augmentation $1-1-1+1=0$, the second triple action kills the first value. This proves the second identity of \eqref{GK99laws} for $r=3$; the first holds because $V$ is abelian. The source theorem applies to these ordinary group identities exactly.
\end{proof}

\begin{theorem}\label{Wspecht}
For every $k\ge1$, $\W_k$ is Specht.
\end{theorem}
\begin{proof}
Consider an ascending chain of fully invariant normal subgroups of $F$. Its images in $F/\Gamma'$ stabilize by Lemma \ref{quotient99}. This quotient is relatively free in $\W_k\cap\Mod(\gamma_3'=1)$, so all its generator substitutions lift to $F$. Its intersections with $\Gamma'$ stabilize by Corollary \ref{centralacc}. Stabilization of image and intersection forces stabilization of the original chain.

Hence every fully invariant subgroup of $F$ is finitely generated as such. Every subvariety of $\W_k$ has a finite relative identity basis. Adding the six defining identities \eqref{W1}--\eqref{W6} makes that basis absolute. This is precisely the Specht assertion.
\end{proof}

\section{The associated bigroup variety}\label{sec:bigroups}
A bigroup is a split group $N\rtimes A$ with a distinguished projection onto $A$; equivalently, it is a group with a distinguished idempotent endomorphism. Let $\LL$ be the bigroup variety defined in \cite{GK01}. In the present notation its conditions are $N,A\in\D$, $N^2\le Z(N\rtimes A)$, centrality of all $r(n;a,b)$, and
\[
[n^{(1+a_1^2)\cdots(1+a_q^2)},n]=1\quad(q\ge1).
\]
\begin{theorem}[Gupta--Krasilnikov]\label{GK01exact}
The bigroup variety $\LL$ just defined is not finitely based, and every proper bigroup subvariety has a finite identity basis. Moreover, if a bigroup identity $u=1$ fails in $\LL$, then the subvariety obtained by adjoining $u=1$ satisfies either $[a_1,a_2]=1$ on the acting group or
\[
[b^{a_1s(a_2)\cdots s(a_m)},b]=1
\]
for some $m\ge1$, with arbitrary kernel input $b$ and acting inputs $a_i$.
\end{theorem}
\begin{proof}[Source and context]
These are the Main Theorem and Proposition 4 of \cite{GK01}, pp.\ 4013--4014. Its variety $V$ is our $\LL$; its relatively free acting group $G$ is our $A$; its $\mathfrak N_{2,4}$ is our $\D$. In its bigroup words, $x$-variables range over the acting group and $y$-variables over the kernel. This sorted notation is equivalent to a group with an idempotent endomorphism $p$: use $p(x)$ for acting inputs and $p(x)^{-1}x$ for kernel inputs. Its identities (1)--(5) are exactly the finite structural conditions displayed above; its family (6) is replaced there by the equivalent square-coefficient family (7), namely our $w_j$, $j\ge1$.

Proposition 4 is proved in Section 3 by reducing a failed identity to an acting word, a square, or one- or two-kernel-variable coefficient expressions, and then extracting a square monomial by independent bigroup substitutions. Section 4 uses this extra diagonal law to deduce all sufficiently long defining laws. Section 5 proves ACC on verbal subbigroups, so that arbitrary proper subvarieties, not merely those defined by one extra law, have finite bases. This explains both the scope and the role of the distinguished projection in the cited theorem. We use it in that signature only. Source identifier: GK01, DOI 10.1081/AGB-100105987; result identifiers: Main Theorem and Proposition 4; no established arXiv identifier.
\end{proof}

\begin{proposition}\label{bigroup-generation}
The bigroup $G_*$ with its natural projection onto $A$ generates $\LL$.
\end{proposition}
\begin{proof}
The group \eqref{groupstar}, with its distinguished projection, belongs to $\LL$: $N^2$ is fixed centrally, the defining relation of $T$ kills the two linear coordinates of $r$, and central involutions $a\in A^2$ are fixed by $\mu$. The Heisenberg commutator formula then verifies all displayed diagonal laws.

The second assertion of Theorem~\ref{GK01exact} forces a proper subvariety of $\LL$ either to have abelian acting group or to satisfy a diagonal law with coefficient $f=a_1s_2\cdots s_m$ (up to renaming). The first possibility fails in $A$. For the second, take the kernel value $(1,1,0)$: the resulting commutator is $f+\mu(f)=a_1s_1\cdots s_m\ne0$ by \eqref{pbw}. Thus $G_*$ generates $\LL$, and therefore
\begin{equation}\label{reductid}
\var_{\rm bgp}(G_*)=\LL,
\qquad \Id(\U)=\Id(\LL)\cap\{\text{pure group words}\}.
\end{equation}
\end{proof}

\begin{proposition}\label{cuts-specht}
The variety $\U$ is contained in $\mathfrak B_4\mathfrak B_2$. For every $k\ge0$, the subvariety $\U_k$ is Specht.
\end{proposition}
\begin{proof}
All reducts of $\LL$ satisfy \eqref{W1}--\eqref{W5}. By \eqref{reductid}, it suffices to check these pure identities in $G_*$. The subgroup $N(R)\rtimes A^2$ has exponent at most four: for $q\in A^2$, the square of $(u,v,c)q$ has linear coordinates $u(1+q),v(1+q)$, whose product under the square formula is zero because $\mu(q)=q$ and $(1+q)^2=0$. The quotient by this subgroup is elementary abelian. This proves $G_*\in\mathfrak B_4\mathfrak B_2$ and exponent at most eight. In this concrete group,
\[
N(R)^2=N(R)'=Z(N(R))=\{(0,0,c):c\in R\}\le Z(G_*).
\]
For example $(c,1,0)^2=(0,0,c)$, and commutation with $(1,0,0)$ and $(0,1,0)$ forces both linear coordinates of a central element to vanish. Also $\gamma_3(G_*)\le N(R)$. Thus \eqref{W2}--\eqref{W3} follow. The projection onto the acting group of $x^4$ is trivial, giving \eqref{W4}. For \eqref{W5}, use $N(R)/N(R)^2$, where inner kernel conjugations disappear and the action factors through $R$ by construction. This last assertion is made for $G_*$, not for an arbitrary member of $\LL$: in general central $r$-values need not belong to $N^2$. Consequently
\begin{equation}\label{UkWk}
\U_k\subseteq\W_k\quad(k\ge1).
\end{equation}
For $k=0$, the argument of Lemma~\ref{quotient99} places $\U_0$ directly in the Specht variety of \cite[Propositions 1--2]{GK99}. For $k\ge1$, apply Theorem~\ref{Wspecht} to \eqref{UkWk}.
\end{proof}

\section{Proper subvarieties and ordinary group polarizations}\label{sec:cofinal}
Throughout this section coefficients are expressions in independent acting variables. They are implemented by ordinary group words as follows. If $c\in R$ has representative $\widehat c=\sum_g\epsilon_g g\in\F[A]$, put
\[
B_c(h_1,h_2;\mathbf x)=\left[\prod_{\epsilon_g=1}h_1^{g(\mathbf x)},h_2\right]
\]
in a fixed order. The variables in $h_1,h_2$ are independent of $\mathbf x$. For $h_1,h_2\in\gamma_3(G_*)\le N(R)$, kernel parts of $g(\mathbf x)$ act trivially modulo $Z(N(R))$. Thus
\begin{equation}\label{paircoordinate}
B_c((\alpha_1,\beta_1,\gamma_1),(\alpha_2,\beta_2,\gamma_2))
=\alpha_1c\mu(\beta_2)+\alpha_2\mu(c)\mu(\beta_1)
\end{equation}
as a central coordinate. This shows that the choice of representative and order does not matter for such a pairing in $G_*$. Every equality of the resulting pure group words is therefore an identity of $\U$. No projection symbol occurs in these words.

\begin{lemma}\label{coordinatewords}
Let $f(x_1,\ldots,x_m)$ be an ordinary group word which is an identity of $A$. Write $x_i=a_i(u_i,v_i,c_i)$. There are coefficients $p_i,d_{ij}\in R$ and $\epsilon_i\in\F$, depending only on $f$ and the acting components, such that
\[
f(\mathbf x)=\left(\sum_i u_i p_i,\sum_i v_i p_i,
\sum_i\epsilon_i c_i+\sum_{i,j}u_i d_{ij}\mu(v_j)\right),
\quad \epsilon_i=\aug(p_i).
\]
If all $p_i=0$, then $\epsilon_i=0$ and $d_{ji}=\mu(d_{ij})$ for every $i,j$.
\end{lemma}
\begin{proof}
In the convention $x=an$, multiplication and inversion are
\[
(an)(a'n')=aa'(n^{a'}n'),\qquad
(an)^{-1}=a^{-1}(n^{-1})^{a^{-1}}.
\]
Induct on the construction of the word from variables, products and inverses. The two row coordinates are linear with the same right coefficients. Central coordinates add; multiplication adds a term $u a'\mu(v')$, and inversion adds $u\mu(v)$. Consequently the central expression always has the displayed bilinear form plus scalar multiples of the original $c_i$. Applying augmentation to a row coefficient gives the exponent sum of that variable modulo two. This is also the scalar coefficient of its central coordinate, proving $\epsilon_i=\aug(p_i)$. The acting part is $f(\mathbf a)=1$ by hypothesis. This induction first gives coefficients for the free acting generators and then gives the same formula under every acting substitution, because $T$ is endomorphism invariant.

Consider
\[
\theta(u,v,c)=(v,u,\mu(c)+v\mu(u)).
\]
Direct multiplication shows that $\theta$ is an involutive automorphism of $N(R)$. Indeed the central coordinate of $\theta(nn')$ reduces to the sum of the central coordinates of $\theta(n),\theta(n')$ plus $v\mu(u')$, exactly the required cross term. Also $\theta(n^a)=\theta(n)^a$ because $a\mu(a)=1$. Extend $\theta$ to $G_*$ by fixing $A$.

Suppose $p_i=0$ for every $i$, and put $Q(\mathbf u,\mathbf v)=\sum_{i,j}u_i d_{ij}\mu(v_j)$. Naturality of word evaluation under $\theta$ gives
\[
Q(\mathbf v,\mathbf u)=\mu(Q(\mathbf u,\mathbf v)).
\]
There are no contributions from the changed $c_i$ because all $\epsilon_i=0$. Set $v_i=u_j=1$ and all other row entries to zero; this gives $d_{ij}=\mu(d_{ji})$. The same evaluation with $i=j$ proves the diagonal case. All these evaluations take place with the free acting generators fixed, so the equalities hold in $R$ itself.
\end{proof}

\begin{lemma}\label{D8reduction}
Every nonabelian locally finite variety of finite-exponent $2$-groups contains $D_8$.
\end{lemma}
\begin{proof}
Choose a finite nonabelian group $H$ in the variety. A finite $2$-group is nilpotent: its center is nontrivial by the class equation, and induction through the central quotient gives a central series. Its quotient by $\gamma_3(H)$ is nonabelian: otherwise $\gamma_2(H)=\gamma_3(H)$, and nilpotence would force $\gamma_2(H)=1$. Thus the variety contains a nonabelian class-two finite $2$-group $K$. Such a group has exponent divisible by four, since an exponent-two group is abelian. Collect any identity of $K$ into individual powers and pairwise commutators. Single-variable evaluations force every power exponent to be divisible by $\exp(K)$, and in particular by four. These powers are therefore identically trivial in $K$. Two-variable evaluations of the remaining collected word at a noncommuting pair force every commutator exponent to be even, since a nontrivial commutator has order a positive power of two. Hence each identity of $K$ holds in $D_8$, as required.
\end{proof}

\begin{proposition}\label{cofinal}
Every proper subvariety $\mathcal X<\U$ is contained in some $\U_k$.
\end{proposition}
\begin{proof}
Choose $f\in\Id(\mathcal X)\setminus\Id(\U)$. If $f$ is not an identity of $\D$, then $D_8\notin\mathcal X$. Lemmas~\ref{heisfacts} and~\ref{D8reduction} imply that $\mathcal X$ is abelian, so $\mathcal X\subseteq\U_0$. We may therefore suppose that $f$ is an identity of $A$ and use Lemma~\ref{coordinatewords}.

First suppose $p_i\ne0$. For independent triple words $h,t$, the words
\[
[f(\mathbf x),t]\qquad\hbox{and}\qquad
[f(x_1,\ldots,x_i h,\ldots,x_m),t]
\]
have central values in $G_*$. Their quotient has value $B_{p_i}(h,t;\mathbf x)$. Indeed the substitution keeps every acting component unchanged and adds the row coordinates of $h$ to those of the $i$th kernel input; the commutator formula uses only those rows. The quotient is a pure consequence of $f=1$, and its equality to $B_{p_i}$ is an identity of $\U$. Thus $\mathcal X$ satisfies a law $B_c(h,t;\mathbf x)=1$ with $0\ne c\in R$.

Suppose instead all $p_i=0$. The value of $f$ is the central coordinate
\[
Q=\sum_i u_i d_{ii}\mu(v_i)+\sum_{i<j}
\bigl(u_i d_{ij}\mu(v_j)+u_j\mu(d_{ij})\mu(v_i)\bigr).
\]
At least one $d_{ij}$ is nonzero, for otherwise $f$ would be an identity of $G_*$. If $d_{ii}\ne0$, take the product of the four word values obtained by replacing $x_i$ by
\[
x_i,\quad x_i h_1,\quad x_i h_2,\quad x_i h_1h_2,
\]
with independent triple words $h_1,h_2$. Since these values are central involutions, this product implements addition of their central coordinates. Linear and other-variable terms cancel. Its value is
\[
\alpha_1d_{ii}\mu(\beta_2)+\alpha_2d_{ii}\mu(\beta_1)
=B_{d_{ii}}(h_1,h_2),
\]
using $d_{ii}=\mu(d_{ii})$. If a mixed coefficient $d_{ij}\ne0$ is selected, use the four independent choices of whether to replace $x_i$ by $x_i h_1$ and $x_j$ by $x_j h_2$. The result is
$\alpha_1d_{ij}\mu(\beta_2)+\alpha_2\mu(d_{ij})\mu(\beta_1)=B_{d_{ij}}(h_1,h_2)$.
The extra central coordinates introduced in multiplying kernel inputs do not enter $Q$. Every operation just described is substitution and multiplication of pure group consequences of $f=1$. Hence again
\begin{equation}\label{bilaw}
B_c(h_1,h_2;\mathbf x)=1\quad\hbox{in }\mathcal X,
\qquad 0\ne c\in R.
\end{equation}

This restriction to triple words does not make $B_c$ an identity of $\U$. In $G_*$ take
\[
h_1=[(1,0,0),a_r,a_s],\qquad h_2=[(0,1,0),a_t,a_u],
\]
where $r,s,t,u$ are four distinct indices absent from $c$. The value is
$(1+a_r)(1+a_s)c\,\mu((1+a_t)(1+a_u))$. In the direct sum \eqref{pbw}, its component with all four new parity bits zero is precisely $c$. Any other summand has a nonzero new parity bit, even after reordering, so the value is nonzero.

Let $J\subseteq R$ be the set of coefficients for which \eqref{bilaw} holds for independent triple words and arbitrary coefficient-variable substitutions throughout $\mathcal X$. We justify all closure operations. Triple values generate $\gamma_3$, and their conjugates are triple values. On $\U$, the pairing is additive in each slot and central, so the law extends to arbitrary elements of $\gamma_3$. Addition of coefficients corresponds to multiplication of these central pairings. For any group word $g(\mathbf x)$ in the coefficient variables, the identities
\[
B_{gc}(h_1,h_2)=B_c(h_1^{g},h_2),\qquad
B_{cg^{-1}}(h_1,h_2)=B_c(h_1,h_2^{g})
\]
hold in $\U$, by conjugation balancing. Conjugating the independent triple-slot variables is an allowed group substitution; hence both products belong to $J$ if $c$ does. Group units span $R$ additively, so $J$ is a two-sided ideal. Finally any endomorphism $a_i\mapsto w_i(\mathbf a)$ of $A$ is realized on coefficients by the pure substitutions $x_i\mapsto w_i(\mathbf x)$, leaving the slot variables untouched. Thus $J$ is $\operatorname{End}(A)$-invariant. This proves the closure assertions without supplying a projection operation to a group in $\mathcal X$.

Write a nonzero member of $J$ as $c=\sum_e a^e p_e(s_1,\ldots,s_n)$ and choose $p_e\ne0$. Put $z_i=a_{n+i}^2$. The endomorphism $\theta_i:a_i\mapsto a_i z_i$ fixes every old $s_j$, and multiplies a parity-$e'$ term by $z_i^{e'_i}$. If the desired $e_i=1$, use $\theta_i(c)+c$; if $e_i=0$, use $\theta_i(c)+cz_i$. In either case the desired parity survives with multiplier $1+z_i=s_{n+i}$ and the other parity vanishes. The operations at distinct $i$ fix each other's new multipliers. After all $n$ operations, multiplying by the inverse of the unit $a^e$ yields
\[
p_e(s_1,\ldots,s_n)s_{n+1}\cdots s_{2n}\in J.
\]
Choose an inclusion-minimal support $E$ of a monomial of $p_e$ and multiply by $\prod_{i\le n,\,i\notin E}s_i$. A different support could survive only if it were a proper subset of $E$, which minimality excludes. The result is $s_1\cdots s_{2n}\in J$. Its pairing word is precisely $b_{2n}$, since successive commutators with squares multiply the first coefficient by $a_i^2-1=1+a_i^2$. Therefore $\mathcal X\subseteq\U_{2n}$. If $n=0$, the nonzero coefficient is $1$ and the same argument gives $b_0=1$.
\end{proof}

\section{Non-finite basability}\label{sec:nonfinite}
We retain the bigroup signature while constructing the witnesses and pass to ordinary group identities only after proving an explicit word calculation.

\begin{lemma}[The source witnesses and their central square]\label{sourcewitness}
For each integer $n>12$ there is a bigroup $H_n=N_n\rtimes A$ generated by $A$ and the $A$-conjugates of one kernel generator $b$, with the following properties. It satisfies $N_n,A\in\D$, centrality of $N_n^2$, and centrality of all $r(d;a,a')$. It satisfies
\[
w_j(d;g_1,\ldots,g_j)=[d^{s(g_1)\cdots s(g_j)},d]=1
\quad(j\ge1,\ j\ne n)
\]
for arbitrary $d\in N_n$, $g_i\in A$. All $(b^g)^2$ equal the same central involution $b^2$. This involution and
\[
c_n=[b^{s_1\cdots s_n},b]
\]
are linearly independent in the central elementary abelian subgroup $N_n^2$. Consequently $\widehat H_n=H_n/\langle b^2\rangle$ has $c_n\ne1$ and satisfies the same structural and excluded-degree identities.
\end{lemma}
\begin{proof}[Source construction and the additional quotient]
Use \cite[Section 2, Steps 1--6; Lemmas 10 and 13 and its corollary; Lemmas 16--18]{GK01}. To specify exactly the construction being used, begin with $H=N_0\rtimes A$, where $N_0$ is relatively free in $\D$ on $b^g$ for $g\in A$. The six subgroups of the source will be denoted $N^{(i)}$. Steps 1 and 2 kill kernel commutators between different $A^2$-cosets and identify the remaining commutators under simultaneous $A$-conjugation. Step 3 identifies $(b^g)^2$ with $b^2$. In particular $N_0^2/N^{(3)}$ is central. Step 4 kills the diagonal commutators coming from $T$: those supported outside $A^2$ were already killed in Step 1, and the central components are precisely the additional relations of Step 4. Step 5 kills the square-monomial diagonals of every degree other than $n$.

The basis asserted at Step 5, p.~4022, is
\[
\{b^2N^{(5)}\}\ \cup\
\{[b^{s_{i_1}\cdots s_{i_n}},b]N^{(5)}:i_1<\cdots<i_n\}
\]
for $N_0^2/N^{(5)}$. Step 6 adjoins all $A$-conjugates of $r(b;a_i,a_j)$. Lemma 13, pp.~4024--4026, states
\begin{equation}\label{witnessintersection}
N^{(6)}\cap N_0^2=N^{(5)}.
\end{equation}
Its proof checks that the natural map from $N^{(6)}/N^{(5)}$ to its image in $N_0/N_0^2$ is injective: Lemmas 14--15 give the same basis, indexed by a basis of $T$, in source and image. Lemmas 11--12 first ensure that this is a map between elementary abelian groups. Thus \eqref{witnessintersection} preserves the entire displayed square-subgroup basis on passing to $H_n=H/N^{(6)}$.

Lemma 10 checks excluded degrees for arbitrary kernel and acting inputs, not just for the named generators. Its proof collects an arbitrary kernel input into conjugates of $b$; paired mixed terms cancel, and the remaining acting coefficient reduces modulo $T$ to monomials of the same degree. Lemmas 16--18 extend the Step 6 relations from basis acting elements to arbitrary inputs and verify centrality of the resulting $r$-values. The completion of Proposition 1, p.~4029, combines these with the earlier steps to give the structural identities. All hypotheses, in particular $n>12$, are retained here.

Finally, by \eqref{witnessintersection}, $c_n$ is a different basis vector from $b^2$. Quotienting by the one-dimensional central subgroup generated by $b^2$ therefore does not kill $c_n$. This last quotient and its use below are our additional step. The source identifier is GK01, DOI 10.1081/AGB-100105987; the precise result identifiers are those given above, with no established arXiv identifier.
\end{proof}

For $k\ge3$ define the ordinary group word
\[
v_k=[[x_1,x_2,x_3],x_1,x_1^2,x_4^2,\ldots,x_k^2,
[x_1^{-1},x_2^{-1},x_3^{-1}]],
\]
where the list $x_4^2,\ldots,x_k^2$ is empty for $k=3$.

\begin{lemma}\label{witnessbridge}
For the quotients of Lemma~\ref{sourcewitness},
\begin{equation}\label{pure-witness}
\widehat H_n\models v_j=1\quad(j\ge3,\ j\ne n),\qquad
v_n(a_1b,a_2,\ldots,a_n)=c_n\ne1.
\end{equation}
\end{lemma}
\begin{proof}
We verify the bridge to \cite{K03}, rather than identifying two differently presented witnesses. In that paper $W=B\rtimes A$, where $B$ is the free class-two group generated by involutions $z^a$, $a\in A$, with the natural shift action. Such a group has exponent dividing four. Let $U_1$ kill commutators $[z^a,z^{a'}]$ with $a(a')^{-1}\notin A^2$; let $U_2$ further kill
\[
[z^a,z^{a'}]^g[z^a,z^{a'}]\quad(a,a',g\in A);
\]
let $U_3$ further kill $[z^t,z]$ for $t\in T$. Finally $D_n$ is generated by $U_3$ and
\[
[z^{s(g_1)\cdots s(g_m)},z]\quad(m\ge1,\ m\ne n,\ g_i\in A).
\]
These definitions are in \cite[Section 2.1, pp.~373--374]{K03}. All the added relations lie in the kernel commutator subgroup and are normal in $W$.

Define a map from $W$ by fixing $A$ and sending $z^a$ to $b^a\langle b^2\rangle$. Kernel class-two relations hold because $N_n\in\D$; generator squares vanish by Step 3 and the quotient by $b^2$; the shift action is preserved on each generator. This defines a surjection $W\to\widehat H_n$. Its kernel contains $U_1$ by Step 1, $U_2$ by Step 2, and $U_3$ by Step 4 of the source construction. Each remaining generator of $D_n$ vanishes by the arbitrary-input assertion of GK01 Lemma 10. Hence we have a surjection
\begin{equation}\label{witnessmap}
W/D_n\longrightarrow\widehat H_n.
\end{equation}
Notice that this argument checks the defining relators before invoking any word identity of $W/D_n$.

Here is the exact external calculation used. For arbitrary $X_i=a_i b_i\in W$, write $b_1=z^{d_1}\pmod{B'}$, $d_1\in\F[A]$, and put $\ell=\aug(d_1)\in\F$. Then \cite[Lemma 3, Corollary 1, and equation (12), pp.~376--381]{K03} gives
\begin{equation}\label{K03exact}
v_k(X_1,\ldots,X_k)D_n
=[z^{s(a_1)\cdots s(a_k)},z]^{\ell}D_n
\quad(k\ge3).
\end{equation}
For clarity, the mechanism in that calculation is as follows. Collect the first triple modulo $B'$ as $z^d$, and the inverse-variable triple as $z^{d''}$. After the intervening commutators, the first slot is
\[
z^{d'},\qquad d'=d(1+a_1)s(a_1)s(a_4)\cdots s(a_k).
\]
Centrality of $B'/U_2$ balances the two slots, so $v_k=[z^{d'\mu(d'')},z]$ modulo $D_n$. Lemma 3 proves in $\F[A]$ that
\[
d'\mu(d'')=\ell\,s(a_1)\cdots s(a_k)
+\sum_h(h+h^{-1})+t,\qquad t\in T,
\]
with a finite sum. The $T$ part vanishes by $U_3$. The paired terms vanish because
$[z^{h^{-1}},z]=[z^h,z]$ modulo $U_2$ and these commutators have exponent two. This proves \eqref{K03exact}. In the proof of Lemma 3, the useful identities are that $s(a)$ and $(1+a)s(a)$ are central modulo $T$, that $s(a)(1+[a,a'])=0$ there, and that $d_1\mu(d_1)$ equals its augmentation plus pairs $h+h^{-1}$. These identities account for the disappearance of the other kernel coefficients; thus the coefficient $\ell$ is exactly the augmentation of $d_1$.

If $k\ne n$, the right side of \eqref{K03exact} is trivial by the definition of $D_n$, proving the corresponding universal identity and its image under \eqref{witnessmap}. For $X_1=a_1z$, $X_i=a_i$ for $i>1$, one has $d_1=1$ and $\ell=1$. Its image is $c_n$, nontrivial by Lemma~\ref{sourcewitness}. This proves both assertions in \eqref{pure-witness}. The source identifier is K03, DOI 10.1112/S0024610703004319, with result identifiers as stated above and no established arXiv identifier.
\end{proof}

\begin{lemma}[Descending-chain compactness]\label{chaincompact}
For a descending sequence of varieties $\mathcal V_1\supseteq\mathcal V_2\supseteq\cdots$ in either of the signatures used here,
\[
\Id\Bigl(\bigcap_N\mathcal V_N\Bigr)=\bigcup_N\Id(\mathcal V_N).
\]
If the intersection has a finite identity basis, it equals some $\mathcal V_N$.
\end{lemma}
\begin{proof}
The identity theories increase. Their union is closed under substitution and every equational inference rule: each application uses finitely many premises, which belong to one common stage. Equivalently, their fully invariant congruences on the absolutely free algebra form a directed union, again a fully invariant congruence. The free algebra modulo that congruence has exactly its prescribed identities, proving that the union is the full theory of its class of models. That model class is the displayed intersection. A finite basis lies at a common stage, forcing equality of the stage with the intersection. The directedness is essential; this assertion is not being used for an arbitrary family of varieties.
\end{proof}

\begin{theorem}\label{nonfinite}
The variety $\U$ is not finitely based.
\end{theorem}
\begin{proof}
For $N>12$ set $\LL_N=\var_{\rm bgp}\{\widehat H_n:n\ge N\}$. The varieties decrease. Every finite structural identity of $\LL$ holds at every stage by Lemma~\ref{sourcewitness}. Each fixed $w_j$ holds at all stages $N>j$. Thus $\LL_\infty=\bigcap_N\LL_N\subseteq\LL$.

If $\LL_\infty$ had a finite basis, Lemma~\ref{chaincompact} would give $\LL_N=\LL_\infty\subseteq\LL$ for some $N$. But $\widehat H_n\in\LL_N$ for every $n\ge N$, and it violates $w_n$, an identity of $\LL$. This is a contradiction. The limit property in Theorem~\ref{GK01exact} therefore forces $\LL_\infty=\LL$.

Let $\mathcal V_N=\var_{\rm grp}\{\widehat H_n:n\ge N\}$. A pure group word is an identity of $\mathcal V_N$ exactly when it is an identity of $\LL_N$. Restricting the increasing union of bigroup identity theories to pure words commutes with that union. Applying Lemma~\ref{chaincompact} in both signatures and \eqref{reductid}, we get
\[
\Id\Bigl(\bigcap_N\mathcal V_N\Bigr)
=\Bigl(\bigcup_N\Id(\LL_N)\Bigr)\cap\{\text{pure words}\}
=\Id(\U).
\]
Hence $\U=\bigcap_N\mathcal V_N$. Each fixed $v_j$, $j\ge3$, eventually holds throughout the tail by Lemma~\ref{witnessbridge}, so it is an identity of $\U$. If $\U$ were finitely based, another application of Lemma~\ref{chaincompact} would yield $\mathcal V_N=\U$. Choosing any $n\ge N$ contradicts the nontrivial evaluation of $v_n$ in \eqref{pure-witness}. This is the required ordinary-group non-finite basability.
\end{proof}

\section{The main theorem}\label{sec:mainproof}
\begin{theorem}\label{main}
Let $A$ and $R$ be the group and algebra defined in Section~\ref{sec:construction}, let $N(R)$ be the group on $R^3$ defined there, and put
\[
G_*=N(R)\rtimes A.
\]
Then $\U=\var(G_*)$ is a limit variety of groups. Moreover, $G_*$ is countable, $\U$ is locally finite, and
\[
\U\subseteq\mathfrak B_4\mathfrak B_2.
\]
In particular, every group in $\U$ has exponent dividing eight.
\end{theorem}
\begin{proof}[Proof of Theorem~\ref{main}]
The construction in Section~\ref{sec:construction} gives a countable group $G_*$. Since $N(R)$ and $A$ belong to $\D$, we have $\U\subseteq\D\D$. The variety $\D$ is locally finite. If a group in $\D\D$ is finitely generated, its quotient in $\D$ is finite, so its normal subgroup in $\D$ has finite index and is finitely generated by Schreier's theorem. That subgroup is consequently finite. Thus $\U$ is locally finite.

Proposition~\ref{cuts-specht} gives $\U\subseteq\mathfrak B_4\mathfrak B_2$, and Theorem~\ref{nonfinite} shows that $\U$ is not finitely based. Let $\mathcal X$ be a proper subvariety of $\U$. By Proposition~\ref{cofinal}, there is a $k\ge0$ such that $\mathcal X\subseteq\U_k$. Proposition~\ref{cuts-specht} states that $\U_k$ is Specht, so $\mathcal X$ has a finite identity basis. Hence $\U$ is a limit variety.
\end{proof}


\begin{thebibliography}{99}

\bibitem{GK99}
C. K. Gupta and A. N. Krasilnikov,
A solution of a problem of Plotkin and Vovsi and an application to varieties of groups,
\textit{J. Austral. Math. Soc. (Series A)} \textbf{67} (1999) 329--355,
\doi{10.1017/S1446788700002056}.

\bibitem{GK01}
C. K. Gupta and A. N. Krasilnikov,
A just non-finitely-based variety of bigroups,
\textit{Comm. Algebra} \textbf{29} (2001) 4011--4046,
\doi{10.1081/AGB-100105987}.

\bibitem{Higman}
G. Higman,
Ordering by divisibility in abstract algebras,
\textit{Proc. London Math. Soc. (3)} \textbf{2} (1952) 326--336,
\doi{10.1112/plms/s3-2.1.326}.

\bibitem{K03}
A. N. Krasilnikov,
A non-finitely based variety of centre-by-abelian-by-nilpotent groups of exponent 8,
\textit{J. London Math. Soc.} \textbf{68} (2003) 371--382,
\doi{10.1112/S0024610703004319}.

\end{thebibliography}
\end{document}